\documentclass[a4paper,12pt,notitlepage]{amsart}
\usepackage{amsfonts, amsmath, amscd, amsthm, graphicx, amssymb,mathtools}
\usepackage{tikz,pgfplots}
\pgfplotsset{compat=1.16}
\usepackage{multicol}
	\usetikzlibrary{patterns,arrows,shapes,decorations.markings,calc}
\usepackage{stackengine}
\usepackage{indentfirst}
\usepackage{mathtools}
\usepackage{tikz-cd} 
\usepackage{enumitem}
\newcommand{\subscript}[2]{$#1 _ #2$}

\usepackage[utf8]{inputenc}

\DeclareFontFamily{U}{min}{}

\DeclareFontShape{U}{min}{m}{n}{<-> udmj30}{}

\usepackage{url} 
\usepackage{mathrsfs}
\usepackage{comment}
\newtheorem{theorem}{Theorem}[section]
\newtheorem*{theorem*}{Theorem}

\newtheorem{proposition}[theorem]{Proposition}
\newtheorem{lemma}[theorem]{Lemma}

\newtheorem{example}[theorem]{Example}
\newtheorem{corollary}[theorem]{Corollary}

\newtheorem*{problem*}{Problem}
\theoremstyle{definition}
\newtheorem{definition}[theorem]{Definition}
\newcommand*\circled[1]{\tikz[baseline=(char.base)]{
            \node[shape=circle,draw,inner sep=2pt] (char) {#1};}}

\makeatletter
\newcommand*{\rom}[1]{\expandafter\@slowromancap\romannumeral #1@}
\makeatother

\newcommand{\id}{\text{id}}
\newcommand{\M}{\mathcal{M}}
\newcommand{\Tsim}{\stackrel T\sim }
\newcommand\stackrqarrow[2]{%
    \mathrel{\stackunder[2pt]{\stackon[4pt]{$\rightsquigarrow$}{$\scriptscriptstyle#1$}}{%
            $\scriptscriptstyle#2$}}}

\title{ The Connectedness Homomorphism between Discrete Morse Complexes}
\author{Chong Zheng}
\address{Faculty of  Science and Engineering, Waseda University, 
Ohkubo, Shinkuju-ku, Tokyo, 169-8555 Japan}
\email{c{\_}zheng@aoni.waseda.jp}
\begin{document}

\begin{abstract}
Given two discrete Morse functions on a 
 simplicial complex, we
  introduce the {\em connectedness homomorphism}
   between the corresponding discrete Morse complexes. 
This concept 
leads to a novel framework 
for studying the connectedness in discrete Morse theory at the 
chain complex level.
In particular, we apply it to
 describe a discrete analogy to `cusp-degeneration' of Morse complexes.
A precise comparison between smooth case and our discrete cases is also given. 

\end{abstract}
\maketitle

\section{Introduction}

 Let $K$ be a finite simplicial complex,
 with $f_1, f_2: K\to \mathbb{R}$ representing two distinct 
\textit{ discrete Morse functions} on $K$.
In our previous  paper \cite{no.1},
we have studied the \textit{(strong) connectedness }
between an $f_1$-critical 
simplex $\tilde{\sigma}_1$  and 
an $f_2$-critical 
simplex  $\tilde{\sigma}_2$,  which characterizes
 the structure of 
 the \textit{gradient vector fields} on $K$
  existing between these two simplices.
 
 Namely,  at simplicial complex level,
we study the uniqueness of 
gradient path, 
the uniqueness and the quantity of 
strong connectedness, and 
their correlation with
the Betti numbers and the Euler number of $K$ \cite{no.1, no.2}.
Note that
 the concept was inspired 
 by a similar one
  in King-Knudson-Mramor Kosta \cite{birth and death}, 
  but there are some crucial differences (see Section 2.3).

In this
paper,
  our primary aim is to establish a
theoretical framework
for understanding the connectedness between
discrete Morse functions
   $f_1$ and  $f_2$
  at the \textit{chain complex level.}
First, we  define the 
 \textit{connectedness homomorphism }
 between two discrete Morse complexes corresponding
 to   $f_1$ and  $f_2$,
 derived from the simplex connectedness 
 between corresponding critical simplices.
 Then, we
 explore some fundamental properties of 
 connectedness homomorphisms.
 In particular,
  using these properties, 
  we give an appropriate description
   of a discrete analogy, 
   called {\em birth-death transition},  
   to the cusp-degeneration 
   of critical points of smooth Morse function in Cerf theory \cite{cerf}.  
In case of smooth Morse theory,
Laudenbach \cite{laudenbach}
described the change of smooth Morse complexes
before and after a cusp-degeneration.
We demonstrate that a 
similar description is 
applicable within the context of discrete Morse functions. 
Hence, 
through Gallais \cite{Gallais} and Benedetti \cite{Benedetti},
our results recover the main results of 
\cite{laudenbach}
but also work for a wider context of 
any simplicial complexes.
This is  one of the initial steps in seeking 
for a discrete analogue to smooth Cerf theory.

In Section 2, 
we provide a
 concise overview of the 
  preliminaries, 
 including 
 the fundamentals of
  discrete Morse theory which are introduced by Forman \cite{Forman_guide, Morse Theory for Cell Complexes},  
  as well as the concepts
   of connectedness and strong connectedness of \cite{no.1}.
In Section 3,
we introduce the connectedness homomorphism 
between two discrete Morse complexes,
and investigate some fundamental properties of it.
In section 4,
we apply the 
connectedness homomorphism 
to describe the  birth-death transition of the discrete Morse complexes,
especially,  in Theorem \ref{A-degeneration},
we prove that 
when the connectedness homomorphisms are birth-death transitions,
 they are chain maps.
Furthermore, 
we establish a clear relationship between 
the cusp-degeneration in smooth  case and that in our discrete case in Proposition 4.5.

\section{Preliminaries}
This section presents 
fundamental concepts for this paper.

\subsection{Discrete Morse Functions and Gradient Vector Fields}
For a comprehensive
understanding 
of  \textit{discrete Morse theory } and \textit{discrete Morse homology},
one may refer
 to  \cite{Forman_guide, Morse Theory for Cell Complexes,Discrete Morse Theory Kozlov, Discrete Morse Theory scoville}.

Let $K$ be a simplicial complex,
$K_p$ be the set of $p$-dimensional simplices.
We use $\sigma^{(p)}$ to represent a 
$p$-dimensional simplex.
When the context is clear, 
we omit the superscript $p$.
Notations $\tau \prec \sigma$  and  $\tau  \succ  \sigma$ are used
to denote
\textit{$\tau $ is a face of $\sigma$} and \textit{$\sigma$ is a face of $\tau $},
respectively.

A \textit{discrete Morse function}
$f:K \to \mathbb{R}$
satisfies for all
$\sigma \in K_p$,
\begin{enumerate}
\item $\displaystyle \# \{\tau^{(p+1)} \succ \sigma \, | \, f(\tau) \leq f(\sigma)\} \leq 1;$
\item 
$\displaystyle  \# \{v^{(p-1)} \prec \sigma \, | \, f(\sigma) \leq f(v)\} \leq 1,$
\end{enumerate}
where $\#$ denotes the cardinality of a set.

We say  a simplex
$\sigma^{(p)}$ is a 
\textit{critical} $p$-dimensional simplex (of $f$)
if
\begin{enumerate}
\item 
$\displaystyle\# \{\tau^{(p+1)} \succ \sigma \, | \, f(\tau) \leq f(\sigma)\} =0;$
\item 
$\displaystyle\# \{v^{(p-1)} \prec \sigma \, | \, f(\sigma) \leq f(v)\} =0.$
\end{enumerate}
If $\sigma$ is a critical simplex,
then we call $f(\sigma)$ a \textit{critical value} of $f$.
Conversely, if a simplex $\sigma$ is not critical, 
then
we say that $\sigma$ is a \textit{non-critical} simplex,
and $f(\sigma)$ is a \textit{non-critical value}.
We define
$$Cr_q(f):= \{ \sigma \in K_q \, | \, \sigma^{(q)} \text{ is } f \text{-critical}  
  \}$$ be the set of $q$-dimensional $f$-critical simplices.

A \textit{gradient vector field} $V$ corresponding to $f$ is 
the collection of pairs 
$(\alpha^{(p)}, \beta^{(p+1)})$
satisfying 
$\alpha \prec \beta$ and
$f(\alpha) \geq f(\beta)$.
We call such a pair $(\alpha, \beta)$
a \textit{gradient pair}.

A  \textit{gradient $V$-path} is a sequence 
of simplices
$$\alpha_0^{(p)}, \beta_0^{(p+1)}, \alpha_1^{(p)}, \beta_1^{(p+1)},\cdots, \beta_r^{(p+1)}, \alpha_{r+1}^{(p)} $$
such that for
each $i= 0, \cdots,r $,
$(\alpha^{(p)}, \beta^{(p+1)})\in V$
and
$\beta_i^{(p+1)} \succ \alpha_{i+1}^{(p)} \neq \alpha_i^{(p)}.$

We denote a \textit{gradient  $V$-path} $\gamma$ 
consisting of $p$ and $(p+1)$-dimensional 
simplices
from 
$\alpha_0^{(p)} $ to 
$\alpha_{r+1}^{(p)}$
by 
$$\gamma: \alpha_0^{(p)}  \stackrqarrow{V(p,p+1)}{}  \alpha_{r+1}^{(p)}. $$
Note that we can also take any $(p+1)$-dimensional simplex
$\beta_0^{(p+1)}$ as the starting simplex, and
$\beta_r^{(p+1)}$ as the ending simplex.
Thus, we can denote the corresponding 
path consisting of $p$ and $(p+1)$-dimensional simplices
by 
$$\gamma: \beta_0^{(p+1)} \stackrqarrow{V(p,p+1)}{}  \beta_r^{(p+1)}. $$
\subsection{Morse Complexes and Discrete Morse Complexes}
For details of Morse theory and 
Morse complexes, 
one may refer to 
\cite{Milnor, Guest}.
Let $M$
 be a compact manifold and 
 $\phi : M \to  \mathbb{R}$ be a Morse function.
The Morse complex
has \textit{chain groups }
$S_i^{\phi}(M)$ as 
the span of the critical points of index $i$.
The \textit{boundary homomorphisms}
$\partial_i^{\phi}: S_i^{\phi} \to S_{i-1}^{\phi}$
are defined in terms of 
the gradient flow lines of $-\nabla \phi$.
More precisely, 
let $p$ be a critical point
of $\phi$ of index $i$,
then
$\partial_i^{\phi}(p):= \sum_{\gamma} m(\gamma) \gamma(p),$
where
$\gamma(p)$ is a flow  from $p$ to a critical point of index $i-1$;
and 
$m(\gamma)=\pm 1$ depending on the 
orientation.
The Morse complex is defined as 
$\{S_{\ast}^{\phi}(M),  \partial^{\phi}_{\ast}\}$,
and is isomorphic to
the ordinary homology.

A discrete version of 
Morse-Smale complex is first constructed by Forman \cite{Forman_guide}.
As a discrete analogy to the Morse complex in ordinary Morse theory,
the $q$-dimensional
chain groups are obtained by spanning
 $q$-dimensional 
critical simplices, and the boundary homomorphisms
are  obtained by the gradient flows between simplices of 
consecutive dimensions.
Precise definitions are given as follows.

From now on, 
we suppose that each simplex of $K$ is oriented.
Then,  given simplices $\alpha^{(q-1)}\prec \sigma^{(q)}$,
we use $[ \sigma : \alpha]$ to represent \textit{the incidence number} of 
 $\sigma$ and $ \alpha$;
$[ \sigma : \alpha]=\pm 1$.

Let $f:K\to \mathbb{R}$ be a discrete Morse function.
We first define the 
\textit{chain group} of discrete Morse complex $C_q^f= C_q^f(K)$
as the free Abelian group generated by
 $\sigma_i \in Cr_q(f)$,  
and
the \textit{$q$-th boundary homomorphism 
$\partial_{q}^{f}: C_q^f \to C_{q-1}^f$ } 
is defined as the linear extension 
of
$$\partial_{q}^{f}(\sigma_i):= \sum_{ \alpha \in Cr_{q-1}(f)} n^{f}(\sigma_i, \alpha) \alpha.$$
Here, $n^{f}(\sigma_i, \alpha)$ is 
given by
$$n^{f}(\sigma_i, \alpha):= \sum_{\gamma } m(\gamma),$$
where $\gamma$ 
runs over all gradient paths $\sigma_i \stackrqarrow{V(q-1,q)}{}\alpha$,
and
$m(\gamma)$ is given by the product of 
each incidence number.
We call $n^{f}(\sigma_i, \alpha)$
the \textit{connectedness coefficient } between  
$\sigma_i$ and  $\alpha$.
In particular, $\partial_{0}^{f}$ is defined as
 the zero homomorphism.

It is known that 
$\partial^f \circ \partial^f=0$ \cite[Theorem 7.1]{Forman_guide}.  
Hence, $\{C_{\ast}^f, \partial_{\ast}^f \}$ 
forms the \textit{discrete Morse chain complex},
and the \textit{discrete Morse homology}
is defined as 
$$H_{\ast}:= \frac{\text{Ker} \partial^f_{\ast} }{\text{Im} \partial^f_{\ast}}.$$
Furthermore, discrete Morse homology is 
isomorphic to the ordinary homology \cite[Theorem 7.1]{Forman_guide}.

\subsection{$\partial$-Connectedness and Strong Connectedness of Critical simplices}

Let $f_1, f_2: K \to \mathbb{R} $ be two different discrete
Morse functions on $K$,
and $V_1$, $V_2$ be the corresponding gradient vector field,
respectively.
For $\tilde{\sigma}_1\in Cr_{q}(f_1)$ and 
$\tilde{\sigma}_2\in Cr_q(f_2)$, 
we  say that 
$\tilde{\sigma}_1$ 
is \textit{$\partial$-connected} to $\tilde{\sigma}_2$ if
\begin{itemize}
\item when $q\neq 0$, 
 there is an $f_1$-gradient path  
$$\tilde{\sigma}_1^{(q)} \stackrqarrow{V_1(q-1,q)}{} \tilde{\sigma}_2^{(q)}; $$

\item when $q=0$, 
there is an $f_2$-gradient path  
$$\tilde{\sigma}_1^{(0)} \stackrqarrow{V_2(0,1)}{} \tilde{\sigma}_2^{(0)}.$$
\end{itemize}

We call 
this relationship \textit{$\partial$-connectedness,}
and denote it 
by $\tilde{\sigma}_1 \to \tilde{\sigma}_2.$
From now, when the context is clear, 
we only call this the \textit{connectedness} between simplices.

Furthermore,
we say that
$\tilde{\sigma}_1$ 
is \textit{strongly connected} to $\tilde{\sigma}_2$
if $\tilde{\sigma}_1$ 
is connected to $\tilde{\sigma}_2$ and 
$\tilde{\sigma}_2$ 
is connected to $\tilde{\sigma}_1$,
denoted  by $\tilde{\sigma}_1 \leftrightarrow \tilde{\sigma}_2$.    
 
We remark that
$0$-dimensional connectedness from 
$ \tilde{\sigma}_1$ to $\tilde{\sigma}_2$ 
involves the gradient paths on 
$V_2$ that is different from other dimensional cases.
Fundamental properties of $0$-dimensional simplex connectedness
are studied in our previous paper \cite{no.1}.
With regard to 
the (strong)
\textit{$\partial$-connectedness,}
several 
properties including the relationship between
the strong connectedness number and 
the Euler characteristic,
are  investigated in \cite[Section 4]{no.1}.

Also, we note that our definition of $\partial$-connectedness is motivated by 
a similar proceeding notion introduced in \cite{birth and death}.
However, it is important to 
note that 
the two definitions are not identical for general cases.
Compared with  the connectedness in \cite{birth and death},
which consists of three consecutive dimensions,  the
$\partial$-connectedness in this paper specifies a ``direction",
thus only two dimensions are involved.

\section{The connectedness homomorphism between discrete Morse complexes}
In this section, we  construct the\textit{ connectedness homomorphism} between  discrete
Morse complexes of  the same simplicial complex.

\subsection{Connectedness Homomorphisms}

Suppose that
we are given discrete Morse functions 
$f_1, f_2:K \to \mathbb{R}$ on a simplicial complex $K$.
Let $m$ be the dimension of $K$.
We denote the 
discrete  Morse complexes by
$\{ C_{\ast}^{f_1}= C_{\ast}^{f_1}(K), \partial_{\ast}^{f_1} \}$ and 
$\{ C_{\ast}^{f_2}= C_{\ast}^{f_2}(K),  \partial_{\ast}^{f_2} \}$ corresponding to 
$f_1$ and $f_2$, respectively.

\begin{definition}

Let $q=1,2 \ldots, m$.
The \textit{$q$-dimensional connectedness homomorphism}
from $C_{q}^{f_1}$ to
$C_{q}^{f_2}$
$$h_q: C_{q}^{f_1} \to C_{q}^{f_2}$$
is defined as the linear extension of 
$$h_q(\tilde{\sigma}_1):= \sum_{\sigma_j\in Cr_q(f_2)} n^{f_1}(\tilde{\sigma}_1, \sigma_j) \sigma_j,$$
where $\tilde{\sigma}_1 \in Cr_q(f_1)$.

In particular,  when $q=0$,
we use $\tilde{v}_1$ to denote a 
$0$-dimensional $f_1$-critical simplex, and 
the \textit{$0$-dimensional connectedness homomorphism}
$h_0: C_{0}^{f_1} \to C_{0}^{f_2}$
is defined as the linear extension of 
$$h_0(\tilde{v}_1):= \sum_{v_j\in Cr_0(f_2)} n^{f_2}(\tilde{v}_1,v_j) v_j.$$

Note that 
the connectedness coefficient $n^{f_2}(\tilde{v}_1,v_j)$
is associated with 
the connectedness from 
$\tilde{v}_1$ to $v_j$, 
which is on the gradient vector field $V_2$ obtained from $f_2$.
Technically,  when $\tilde{\sigma}_1=\tilde{\sigma}_2$,
we let $$n^{f_1}(\tilde{\sigma}_1,\tilde{\sigma}_2)=n^{f_2}(\tilde{\sigma}_2,\tilde{\sigma}_1)=1.$$

On the other hand, 	
we can also define the connectedness homomorphism
from $C_q^{f_2}$ to
$C_q^{f_1}$
with a parallel way 
as 
$$g_q: C_{q}^{f_2} \to C_{q}^{f_1}.$$
Thus, 
we can use the  diagram in Figure \ref{diagram} to 
denote the connectedness homomorphisms and 
the boundary
homomorphisms of the two chain
complexes.

\end{definition}

\begin{figure}

\begin{center}

\begin{tikzcd}
{\vdots} \arrow[dd]             &  &  & {\vdots} \arrow[dd]             &  &  & {\vdots} \arrow[dd] \\
                          &  &  &                           &  &  &               \\
{C_{q+1}^{f_1}} \arrow[rrr,"h_{q+1}"] \arrow[dd,"\partial_{q+1}^{f_1}"] &  &  & {C_{q+1}^{f_2}} \arrow[dd,"\partial_{q+1}^{f_2}"] \arrow[rrr,"g_{q+1}"] &  &  & {C_{q+1}^{f_1}} \arrow[dd,"\partial_{q+1}^{f_1}"]\\
                          &  &  &                           &  &  &               \\
{C_q^{f_1}} \arrow[rrr,"h_q"] \arrow[dd,"\partial_{q}^{f_1}"] &  &  & {C_q^{f_2}} \arrow[dd,"\partial_{q}^{f_2}"] \arrow[rrr,"g_q"] &  &  & {C_q^{f_1}} \arrow[dd,"\partial_{q}^{f_1}"]\\
                          &  &  &                           &  &  &               \\
{C_{q-1}^{f_1}} \arrow[rrr,"h_{q-1}"]   \arrow[dd]         &  &  & {C_{q-1}^{f_2}} \arrow[rrr,"g_{q-1}"]   \arrow[dd]         &  &  & {C_{q-1}^{f_1}} \arrow[dd] \\
                          &  &  &                           &  &  &               \\
{\vdots}                        &  &  & {\vdots}                        &  &  & {\vdots}         
\end{tikzcd}
\end{center}

\caption{Connectedness homomorphisms and the boundary homomorphisms.}

\label{diagram}
\end{figure}
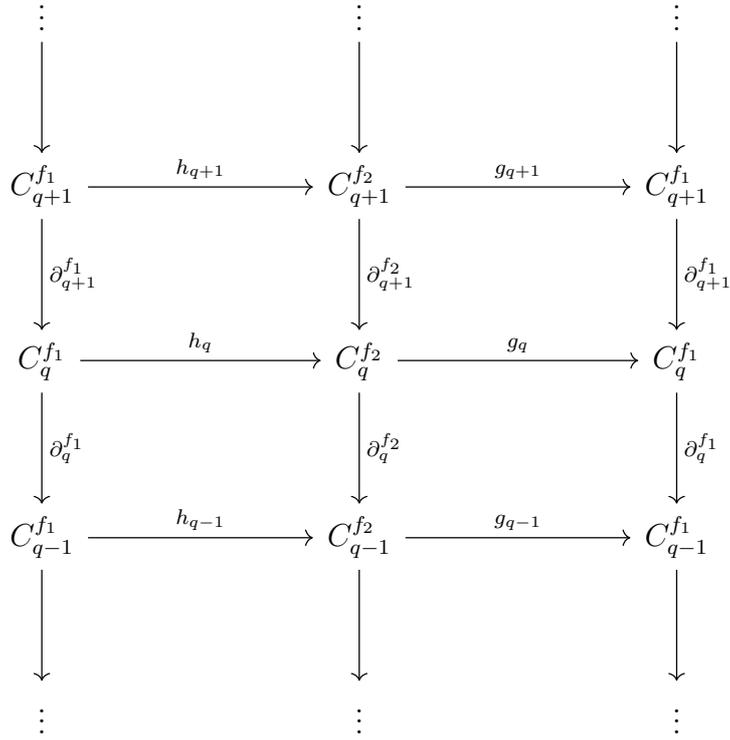

Regarding
to the connectedness
homomorphism, it is natural to ask the follows:
\begin{enumerate}

\item Under what conditions $h=\{h_q\}_q$ and $g=\{g_q\}_q$
are  \textit{chain maps}? 
 $h$ is called a \textit{chain map} if 
for any $q>0$,
$\partial^{f_2}_q \circ h_q= h_{q-1}\circ \partial^{f_1}_q$. 
It is equivalent to that
  the diagram in Figure 1 is commutative.
\item  Under what conditions $h$ and $g$ are inverse  to each other, 
i.e.
$h\circ g= g\circ h =\id$? 
 
 \item Under what conditions $h$ and $g$ are  \textit{faithful} connectedness homomorphisms?
(See Definition \ref{faithful})
\end{enumerate}

In the following subsections,
we give and explain partial solutions to those questions.
A more general answer to those questions 
requires a comprehensive understanding
of connectedness at the simplicial level, 
thus
more discussion
will be given in \cite{no.2}.

We begin with  the following example showing that not every 
connectedness homomorphism is a chain map.

\begin{example}
Figure \ref{an example}
illustrates  discrete Morse functions $f_1$ and $f_2$
on a graph $G$.
Red vertices and edges denote critical simplices 
 with respect to 
each discrete Morse function.
Arrows from $0$-dimensional simplices to
$1$-dimensional simplices 
represent the gradient pairs.

We can compute,  modulo $2$, that 
$\partial^{f_1} \circ g_1 (e_2^1)= v_1^1 + v_1^2$;
and
$g_0\circ \partial^{f_2} (e_2^1) = 0$.
Therefore, 
the connectedness homomorphism $g$ is  
not a chain map.
\end{example}

\begin{figure}[ht]
\begin{minipage}{\textwidth}

  \centering

\begin{tikzpicture}[>=stealth, thick, scale=1.25,
  decoration={markings, mark=at position 0.5 with {\arrow{>}}}
]
\foreach \i in {1,...,6} {
    \pgfmathsetmacro{\angle}{360/6*\i}
    \def\nodecolor{black}
    \ifnum\i=1
      \def\nodecolor{black}
    \fi
    \node[draw, circle, fill=\nodecolor, inner sep=2pt] (A\i) at (\angle:2) {};
  
  }
  
   \draw[red] (A1) -- (A2) node[midway, above,black] {$e_1^1$};  
    \draw[postaction={decorate}] (A2) -- (A3);
\draw[postaction={decorate}] (A3) -- (A4);
  \draw[postaction={decorate}] (A4) -- (A5);
\draw[postaction={decorate}] (A6) -- (A5);
 \draw[postaction={decorate}] (A1) -- (A6);

  \path let \p1 = (A2) in \pgfextra{\xdef\yTwo{\y1}};
  \path let \p1 = (A4) in \pgfextra{\xdef\yFour{\y1}};

  \foreach \i in {10,11,12} {
    \pgfmathsetmacro{\xcoord}{-7 + (\i - 9) * 1.5} 
    \def\nodecolor{black}
    \ifnum\i=10
      \def\nodecolor{red} 
      \node at (\xcoord, \yFour - 0.3cm) {$v_1^1$};
    \fi
    \node[draw, circle, fill=\nodecolor, inner sep=2pt] (A\i) at (\xcoord,\yFour) {};
  }

  \draw[red] (A10) -- (A11) node[midway, above,black] {$e_1^2$};  
  \draw[postaction={decorate}] (A11) -- (A12);
  \draw[postaction={decorate}] (A12) -- (A4); 
  
  \foreach \i in {13,14,15} {
    \pgfmathsetmacro{\xcoord}{-3.5 + (\i - 9) * 1.5} 
    \def\nodecolor{black}
    \ifnum\i=15
      \def\nodecolor{red} 
      \node at (\xcoord, \yFour - 0.3cm) {$v_1^2$};
    \fi
    \node[draw, circle, fill=\nodecolor, inner sep=2pt] (A\i) at (\xcoord,\yFour) {};
  }

  \draw[postaction={decorate}] (A14) -- (A15);
  \draw[postaction={decorate}] (A13) -- (A14);
  \draw[postaction={decorate}] (A5) -- (A13); 
\end{tikzpicture}
(a) Graph $G$ with discrete Morse function $f_1$.

\vspace{0.5cm}

\end{minipage}

\begin{minipage}{\textwidth}
  \centering

\begin{tikzpicture}[>=stealth, thick, scale=1.25,
  decoration={markings, mark=at position 0.5 with {\arrow{>}}}
]
\foreach \i in {1,...,6} {
    \pgfmathsetmacro{\angle}{360/6*\i}
    \def\nodecolor{black}
    \ifnum\i=1
      \def\nodecolor{black}
    \fi
    \node[draw, circle, fill=\nodecolor, inner sep=2pt] (A\i) at (\angle:2) {};
  
  }
  
    \draw[postaction={decorate}] (A1) -- (A2);
    \draw[postaction={decorate}] (A2) -- (A3);
\draw[postaction={decorate}] (A3) -- (A4);
  \draw[red] (A4) -- (A5) node[midway, above,black] {$e_2^1$};  
\draw[postaction={decorate}] (A5) -- (A6);
 \draw[postaction={decorate}] (A6) -- (A1);

  \path let \p1 = (A2) in \pgfextra{\xdef\yTwo{\y1}};
  \path let \p1 = (A4) in \pgfextra{\xdef\yFour{\y1}};

  \foreach \i in {10,11,12} {
    \pgfmathsetmacro{\xcoord}{-7 + (\i - 9) * 1.5} 
    \def\nodecolor{black}
    \ifnum\i=10
      \def\nodecolor{red} 
      \node at (\xcoord, \yFour - 0.3cm) {$v_2^1$};
    \fi
    \node[draw, circle, fill=\nodecolor, inner sep=2pt] (A\i) at (\xcoord,\yFour) {};
  }

 \draw[postaction={decorate}] (A11) -- (A10);
  \draw[postaction={decorate}] (A12) -- (A11);
  \draw[postaction={decorate}] (A4) -- (A12); 
  
  \foreach \i in {13,14,15} {
    \pgfmathsetmacro{\xcoord}{-3.5 + (\i - 9) * 1.5} 
    \def\nodecolor{black}
    \ifnum\i=15
      \def\nodecolor{black} 
    \fi
    \node[draw, circle, fill=\nodecolor, inner sep=2pt] (A\i) at (\xcoord,\yFour) {};
  }

  \draw[postaction={decorate}] (A15) -- (A14);
  \draw[postaction={decorate}] (A14) -- (A13);
  \draw[postaction={decorate}] (A13) -- (A5); 
\end{tikzpicture}

(b) Graph $G$ with discrete Morse function $f_2$.
\end{minipage}
\caption{ Connectedness homomorphism $g$ is not a chain map.}
\label{an example}
\end{figure}

Now, we directly computing
the equation of chain map by 
employing the definition.
The following
 fundamental but important  necessary and sufficient  condition 
to make $h$ a chain map is 
straightforward by simple computations.

\begin{proposition}
The following $(1)$ and $(2)$ are equivalent:
\begin{enumerate}

\item $h$ is a chain map.
\item For any  $q> 1$ and any $\tilde{\sigma}_1\in Cr_q(f_1)$,
\begin{gather*}
 \sum_{\sigma_j \in Cr_q(f_2)}  n^{f_1}(\tilde{\sigma}_1, \sigma_j )\sum_{\alpha_i \in Cr_{q-1}(f_2)} n^{f_2}(\sigma_j , \alpha_i) \alpha_i= \\
 \sum_{\alpha_{\ell} \in Cr_{q-1}(f_1)} n^{f_1}(\tilde{\sigma}_1, \alpha_{\ell}) \sum_{\alpha_{i} \in Cr_{q-1}(f_2)} n^{f_1}(\alpha_{\ell}, \alpha_i) \alpha_i;
 \end{gather*}
\end{enumerate}
and for
$q=1,$
\begin{gather*}
 \sum_{\sigma_j \in Cr_1(f_2)}  n^{f_1}(\tilde{\sigma}_1, \sigma_j )\sum_{\alpha_i \in Cr_{0}(f_1)} n^{f_2}(\sigma_j , \alpha_i) \alpha_i= \\
 \sum_{\alpha_{\ell} \in Cr_{0}(f_1)} n^{f_1}(\tilde{\sigma}_1, \alpha_{\ell}) \sum_{\alpha_{i} \in Cr_{0}(f_2)} n^{f_2}(\alpha_{\ell}, \alpha_i) \alpha_i.
 \end{gather*}

\end{proposition}

\begin{proof}
We calculate both sides for any $q>1$ and $\tilde{\sigma}_1 \in Cr_q(f_1)$:
\begin{align*}
\partial^{f_2}\circ h_q(\tilde{\sigma}_1)&=
 \partial^{f_2}(  \sum_{\sigma_j \in Cr_q(f_2)} n^{f_1}(\tilde{\sigma}_1,\sigma_j) \sigma_j) \\
&=  \sum_{\sigma_j \in Cr_q(f_2)} n^{f_1}(\tilde{\sigma}_1, \sigma_j)\partial^{f_2}(\sigma_j) \\
&= \sum_{\sigma_j \in Cr_q(f_2)} n^{f_1}(\tilde{\sigma}_1, \sigma_j)\sum_{\alpha_i \in Cr_{q-1}(f_2)} n^{f_2}(\sigma_j , \alpha_i) \alpha_i,
\end{align*} 

and

\begin{align*}
h_{q-1}\circ \partial^{f_1}(\tilde{\sigma}_1)&= h_{q-1}(\sum_{\alpha_{\ell} \in Cr_{q-1}(f_1)} n^{f_1}(\tilde{\sigma}_1, \alpha_{\ell}) \alpha_{\ell}) \\
&=  \sum_{\alpha_{\ell} \in Cr_{q-1}(f_1)} n^{f_1}(\tilde{\sigma}_1,\alpha_{\ell})h_{q-1}(\alpha_{\ell}) \\
&= \sum_{\alpha_{\ell} \in Cr_{q-1}(f_1)} n^{f_1}(\tilde{\sigma}_1, \alpha_{\ell}) \sum_{\alpha_i \in Cr_{q-1}(f_2)} n^{f_1}(\alpha_{\ell}, \alpha_i) \alpha_i.
\end{align*}

We omit the proof for the case $q=1$,
as it can be derived through similar computations.

\end{proof}


\subsection{ Connectedness and Connectedness Homomorphism}
In this subsection, we 
discuss the relationship between 
the connectedness of critical simplices and the connectedness homomorphism.

The connectedness between two 
critical simplices only depends on
the presence of discrete gradient flows.
In contrast,
the derivation 
of the connectedness homomorphism, 
as informed by the connectedness coefficient, 
needs 
not only the verification  of connectedness 
but also an understanding of the quantity
 and the orientation of such connectedness.
Hence, it is notably to remark 
that  the relationship 
between simplex connectedness and 
the connectedness homomorphism 
is significantly influenced by
 the choice of the coefficient in the calculation
 of the connectedness 
coefficient.

The number of 
gradient paths connecting 
$\sigma_1$ to $\sigma_2$ is 
essential for determining the connectedness coefficient 
$n^{f_1}(\sigma_1, \sigma_2)$.
When there are multiple 
such gradient paths,
those paths form 
cycles within the simplicial complex.
Although the isomorphism 
between 
discrete Morse homology and
ordinary (simplicial) homology 
is natural,
the geometric information 
contained in those two kinds of 
cycles are different.
A discrete Morse cycle always 
contains a simplicial cycle.
However,
a simplicial cycle may not
contain a discrete Morse cycle.

The following lemma
shows a fundamental relationship between 
those cycles and $\sigma_2$.

\begin{lemma}
Suppose that the $V_1$-gradient paths connecting
 $\sigma_1 \in Cr_q(f_1)$ to $\sigma_2\in Cr_q(f_2)$
 form cycles within $K$.
Then,  $\sigma_2$ does not participate in any cycle.

\end{lemma}

\begin{proof}
A straightforward fact from 
the definition of connectedness is 
that 
if $\sigma_1$ is connected to $\sigma_2$ and $\sigma_1 \neq \sigma_2$,
then $\sigma_2$
is not critical with respective to $f_1$.
Thus,  by the definition
of discrete Morse function,
there exists only one $\alpha_2^{(q-1)} \prec \sigma_2$
that is in the $V_1$-gradient pair $(\alpha_2, \sigma_2)$.
 Consequently, 
 $\sigma_2$ can not participate in any of the cycles.

\end{proof}

Since for each gradient path
$$\gamma: \sigma_1 \stackrqarrow{V_1(q-1,q)}{} \sigma_2,$$
$m(\gamma)$ equals $\pm 1$,
the parity (odd or even) 
of the number of such gradient paths
depends on the connectedness coefficient.
More precisely,
\begin{itemize}
\item $n^{f_1}(\sigma_1, \sigma_2)=0$ implies
the number of gradient paths is even;
and 

\item $|n^{f_1}(\sigma_1, \sigma_2)|=1$ implies
the number of gradient paths is odd,
where $|\cdot|$
denotes the absolute value.

\end{itemize}

On the contrary,
the validity of 
these statements may be false due to 
  the potential presence of torsion in $K$.
However, 
 if calculations are performed using coefficient
$\mathbb{Z}_2$,
then the converse of above statements holds true.
On the other hand,
when the homology of $K$ consists of
torsion element,
the connectedness coefficient 
 $|n^{f_1}(\sigma_1, \sigma_2)|$ 
may equal to $k>1$.

Also, $n^{f_1}(\sigma_1, \sigma_2)\neq 0$ implies 
$\sigma_1$ is connected to $\sigma_2$.  However, the converse does not 
hold.

\subsection{Faithful Homomorphisms }

In further investigations and applications of 
the connectedness between $f_1$ and $f_2$,
it is necessary to first establish
  both connectedness homomorphisms
  $h$ and $g$  as chain maps. 
The concepts of \textit{weak faithfulness} 
and \textit{faithfulness} are introduced
to create a locally one-to-one correspondence of 
simplex connectedness,
which is the case when $\tilde{\sigma}_1$ is only connected to 
one $\tilde{\sigma}_2\in Cr_q(f_2)$
in Proposition 3.2 (2).

\begin{definition}
\label{weakly faithful}
Let $\tilde{\sigma}_1\in Cr_q(f_1)$ and $\tilde{\sigma}_2\in Cr_q(f_2)$.
We say $\tilde{\sigma}_1$ is \textit{weakly faithful}
to  $\tilde{\sigma}_2$
if  
$$h_q(\tilde{\sigma}_1)= n^{f_1} (\tilde{\sigma}_1, \tilde{\sigma}_2) \tilde{\sigma}_2,$$
where $ n^{f_1} (\tilde{\sigma}_1, \tilde{\sigma}_2)\neq 0$.

\end{definition}

Note that when $q=0$,
the connectedness coefficient is
associated with $f_2$.

Let $\tilde{\sigma}_1\in Cr_q(f_1)$ and $\tilde{\sigma}_2\in Cr_q(f_2)$ 
be weakly faithful to each other.
When the computations are performed with 
coefficient $\mathbb{Z}_2$,
then
$n^{f_1} (\tilde{\sigma}_1, \tilde{\sigma}_2)= n^{f_2} (\tilde{\sigma}_2, \tilde{\sigma}_1)=1.$

Next, we define a stronger correspondence 
between $\tilde{\sigma}_1$ and $\tilde{\sigma}_2$
than the weak faithfulness
as follows.

\begin{definition}
\label{faithful}
Let $\tilde{\sigma}_1\in Cr_q(f_1)$ and $\tilde{\sigma}_2\in Cr_q(f_2)$.
We say 
$\tilde{\sigma}_1$ is \textit{ faithful}  to
 $\tilde{\sigma}_2$
 if  $\tilde{\sigma}_1$ is weakly faithful to  $\tilde{\sigma}_2$,
and
 $$\sum_{\alpha_1 \in Cr_{q-1}(f_1)} n^{f_1} (\tilde{\sigma}_1, \alpha_1  )  \sum_{\alpha_2 \in  Cr_{q-1} (f_2)} n^{f_1}( \alpha_1 , \alpha_2) \alpha_2  =  \sum_{\alpha_2 \in  Cr_{q-1} (f_2)} n^{f_1}( \tilde{\sigma}_1, \tilde{\sigma}_2) 
n^{f_2} (\tilde{\sigma_2} ,  \alpha_2)  \alpha_2.$$
If for  any
$\tilde{\sigma}_1\in Cr_q(f_1)$,
 there 
exists 
$\tilde{\sigma}_2\in Cr_q(f_2)$
such that 
$\tilde{\sigma}_1$ is faithful to 
 $\tilde{\sigma}_2$,
 then we call
 $h_q$  a \textit{faithful homomorphism}.
 Furthermore,
if  any $q\in \mathbb{N}$, 
$h_q$ is faithful,
then we say $h$ is \textit{faithful}.

\end{definition}

With regard to 
the zero-dimensional 
connectedness homomorphism 
$h_0$,
we present the following proposition
to show a basic property of $h_0$.

\begin{proposition}
$h_0$ is a faithful homomorphism.
\end{proposition}
\begin{proof}
By the definition of discrete Morse function,
it is easy to see that
for any  vertex $v_1 \in Cr_0(f_1)$,
there exists  a unique corresponding vertex 
$v_2 \in Cr_0(f_2)$ such that 
$v_1$ is connected to $v_2$,
implying
$h_0(v_1)= \pm v_2$.
Therefore, $h_0$ is faithful.
\end{proof}

The above proof  also demonstrates that
$h_0$ is injective.
Furthermore,
under the consideration of coefficient $\mathbb{Z}_2$,
$h_0$ can be 
regarded as  a 
well-defined mapping 
between the  sets  of critical vertices
$Cr_0(f_1)$ and $Cr_0(f_2)$.

\begin{proposition}
\label{faithful implies chain map}
Let $f_1, f_2: K \to \mathbb{R}$
be discrete Morse functions
and $h_q: C_q^{f_1}(K) \to C_q^{f_2}(K) $
be the connectedness homomorphisms.
If $h$ is faithful, 
then $h$ is a chain map. 
\end{proposition}

\begin{proof}
For any 
$\tilde{\sigma}_1\in Cr_q(f_1)$,
suppose that $\tilde{\sigma}_1$ is faithful to $\tilde{\sigma}_2$.
We compute that 
$$\partial_{q}^{f_2} \circ h_q (\tilde{\sigma}_1) =\sum_{\alpha_2 \in  Cr_{q-1} (f_2)} n^{f_1}( \tilde{\sigma}_1, \tilde{\sigma}_2) 
n^{f_2} (\tilde{\sigma_2} ,  \alpha_2)  \alpha_2. $$
Also, 
$$h_{q-1}\circ \partial_{q}^{f_1}(\tilde{\sigma}_1)=
\sum_{\alpha_1 \in Cr_{q-1}(f_1)} n^{f_1} (\tilde{\sigma}_1, \alpha_1  )  \sum_{\alpha_2 \in  Cr_{q-1} (f_2)} n^{f_1}( \alpha_1 , \alpha_2) \alpha_2.$$

Therefore, by Definition \ref{faithful},
$h$ is a chain map.

\end{proof}

\subsection{The Connectedness Homomorphism of Optimal Discrete Morse functions }
The concept of  faithfulness is introduced
to establish a local one-to-one correspondence between
two gradient vector fields similar to 
the optimal discrete Morse function case.
Recall that,
a discrete Morse function $f: K \to \mathbb{R}$
is called \textit{optimal}
if for any $q\in \mathbb{N}$,
$\# Cr_q(f)=\beta_q(K)$ \cite{Discrete Morse theory on graphs}.

\begin{lemma}
\label{optimal then wf}
Let $f_1, f_2 : K\to \mathbb{R}$ 
be optimal discrete Morse functions.
For any $\tilde{\sigma}_1 \in Cr_q(f_1)$,
there exists a unique $\tilde{\sigma}_2\in Cr_q(f_2)$
such that 
$\tilde{\sigma}_1$ and $\tilde{\sigma}_2$
are weakly faithful to each other.

\end{lemma}

\begin{proof}
First,
we note that 
any $q$-dimensional representative chain  (of simplicial homology)
  consists of
at least one $q$-dimensional 
critical simplex.
This fact simply follows
from the isomorphism 
between discrete Morse 
homology and simplicial
homology,
and the discrete Morse inequalities.

Since $f_1$ is optimal,
$\tilde{\sigma}_1$ corresponds to 
a $q$-dimensional homology class $[h]$.
The $f_1$-gradient flows starting from 
$\tilde{\sigma}_1$ cover the whole cycle and flow to  
every $q$-dimensional simplex in the cycle only once.
Additionally, there exists a unique 
$\tilde{\sigma}_2 \in Cr_q(f_2)$
in the  cycle, because $f_2$ is optimal.
Thus, 
$$h_q(\tilde{\sigma}_1)=n^{f_1}(\tilde{\sigma}_1,\tilde{\sigma}_2) \tilde{\sigma}_2=\pm\tilde{\sigma}_1.$$

A parallel discussion to $g_q(\tilde{\sigma}_2)$
completes the proof.

\end{proof}

\begin{proposition}
If $f_1, f_2 : K\to \mathbb{R}$ 
are optimal discrete Morse functions,
then $h$ and $g$ are faithful homomorphisms.
\end{proposition}

\begin{proof}
When $f_1$ is optimal,
for any $\tilde{\sigma}_1 \in Cr_q(f_1)$,
the Morse boundary 
$$\partial^{f_1}_q (\tilde{\sigma}_1)= \sum_{\alpha_i \in Cr_{q-1}(f_1)} n^{f_1} (\tilde{\sigma}_1, \alpha_i) \alpha_i=0.$$
Thus,
the coefficient 
$n^{f_1} (\tilde{\sigma}_1, \alpha_i)=0$, for each 
$\alpha_i \in Cr_{q-1}(f_1)$.

According to Lemma \ref{optimal then wf},
there exists 
$\tilde{\sigma}_2 \in Cr_q(f_2)$ such that 
$\tilde{\sigma}_1$ is weakly faithful to $\tilde{\sigma}_2$.
Thus, 
we compute 
$\partial_q^{f_2} \circ h_q (\tilde{\sigma}_1)= h_{q-1} \circ \partial_q^{f_1} (\tilde{\sigma}_1)=0.$

Therefore, 
$h$ is a faithful homomorphism.
\end{proof}

\begin{corollary}
Let $f_1$ and $f_2$ be optimal discrete Morse functions.
The following statements hold true:
\begin{enumerate}
\item  $h$ and $g$ are chain maps;

\item  $h\circ g=g\circ h=\id$.

\end{enumerate}

\end{corollary}

The connectedness homomorphisms between
optimal functions are
particular cases of faithfulness connectedness homomorphisms.
Generally,
if $\tilde{\sigma}_1$ and $\tilde{\sigma}_2$
are faithful to each other, then
the connectedness coefficients
$n^{f_1}(\tilde{\sigma}_1,\tilde{\alpha}_1)$
and 
$n^{f_2}(\tilde{\sigma}_2,\tilde{\alpha}_2)$
are 
$0$, $1$ or $k>1$.
When 
discrete Morse functions 
$f_1$, $f_2$ are optimal,
every 
connectedness coefficient
$n^{f_1}(\tilde{\sigma}_1,\tilde{\alpha}_1)$ equals $0$.


\section{Modification of the Discrete Morse complex near a birth-death
point via the connectedness homomorphism}

While discrete 
Morse theory and
 ordinary Morse theory 
 share many analogous 
 concepts, the definition 
 of a critical simplex within 
 discrete Morse theory presents 
 unique challenges, especially 
 in the context of depicting a ``degenerate critical simplex".
In ordinary Morse theory,
 the analysis of a degenerate critical point
often
  involves  a \textit{perturbation} strategy. 
A perturbation strategy aims to
 split the degenerate critical point
  into a pair of distinct, 
  non-degenerate critical points,
  thus facilitating further investigation.
  We seek for a discrete analogy.
  
  \subsection{Birth-Death Transition of Discrete Morse Complexes}

Now,
 let 
$f_1, f_2$ be two discrete Morse 
functions
different from a
pair of 
birth-death point.
We then
 apply the connectedness homomorphisms between
$\{C_{\ast}^{f_1}(K), \partial_{\ast}^{f_1} \}$ and 
$\{C_{\ast}^{f_2}(K), \partial_{\ast}^{f_2} \}$
 to 
 feature a discrete version of the  degeneration of critical points in 
the smooth case.
This is a fundamental yet natural
scenario  to explore 
before we consider about
the elementary 
cobordism of
discrete Morse theory.

 We initiate our exploration by 
 defining the context in terms of discrete Morse theory. 
 This involves 
assuming  that the 
chain complexes
$\{C_{\ast}^{f_1}(K), \partial_{\ast}^{f_1} \}$ and 
$\{C_{\ast}^{f_2}(K), \partial_{\ast}^{f_2} \}$ satisfy the 
following assumptions:

We fix a dimension $i$.
 Regarding to the number of $i$-dimensional simplices, 
we
assume 
$$\#Cr_i(f_1)= \# Cr_i(f_2) -1; $$
$$\#Cr_{i-1}(f_1)=\# Cr_{i-1}(f_2) -1.$$
For dimension $q \neq i, i-1$, 
we   
assume that 
$$\#Cr_q(f_1)= \#Cr_q(f_2).$$
We denote the two redundant
simplices 
$\tilde{\sigma}\in Cr_i(f_2)$
and
$\tilde{\alpha}\in Cr_{i-1}(f_2)$,
and suppose that 
$$n^{f_2}(\tilde{\sigma},\tilde{\alpha})=k,$$
where $k\in \mathbb{Z}$.
 
Regarding to the connectedness homomorphisms,
we assume that 
for any  $\delta_1\in Cr_q(f_1)$, 
$$h_q(\delta_1)= 
\begin{cases}
\delta_1,  & \text{ if }q\neq i;\\
\delta_1+n^{f_1}(\delta_1, \tilde{\sigma}) \tilde{\sigma}, &\text{ if } q=i.

\end{cases}$$

On the other hand,
for any  $f_2$-critical simplex $\delta_2$, 
we let
$$g_q(\delta_2)= 
\begin{cases}
\delta_2, & \text{ if }\delta_2 \neq \tilde{\sigma}, \tilde{\alpha};\\
0, &\text{ if } \delta_2=\tilde{\sigma};\\
\displaystyle -\sum_{\alpha\in Cr_{i-1}(f_1)} n^{f_2} (\tilde{\alpha}, \alpha) \alpha, & \text{ if }  \delta_2=\tilde{\alpha}.

\end{cases}$$

  Regarding to the boundary homomorphisms,
we assume the following equations (1) to (5) 
hold.

Namely, when  $ q\neq i, i+1   \text{ and } 
\delta\neq \tilde{\alpha},$
\begin{equation}\label{1}
\partial^{f_1}_q (\delta)= \partial^{f_2}_q (\delta).
\end{equation}

For any  $\tau \in Cr_{i+1}(f_2)$,
\begin{equation} \label{2}
\partial_{i+1}^{f_2}(\tau )=  \partial_{i+1}^{f_1}(\tau ) + \sum_{\sigma \in Cr_{i}(f_1)} n^{f_1} (\tau, \sigma)\cdot n^{f_1}(\sigma, \tilde{\sigma})  \tilde{\sigma}.
\end{equation}
 For any $\sigma \in Cr_{i}(f_2)$ and $\sigma \neq \tilde{\sigma}$,
 \begin{equation} \label{3}
\partial_{i}^{f_2}(\sigma )=\partial_i^{f_1}(\sigma) - n^{f_1}(\sigma, \tilde{\sigma})   (k \tilde{\alpha} + \sum_{\alpha \in Cr_{i-1}(f_1)   }    k \cdot n^{f_2} (\tilde{\alpha}, \alpha   )  \alpha     ).
\end{equation}

 \begin{equation} \label{4}
 \partial_{i}^{f_2}(\tilde{\sigma })=  k \tilde{\alpha}  +   \sum_{\alpha \in Cr_{i-1}(f_1)}  k\cdot n^{f_2} (\tilde{\alpha},  \alpha) \alpha  .
\end{equation}

 \begin{equation} \label{5}
  \partial_{i-1}^{f_2}(\tilde{\alpha })= \sum_{\alpha \in Cr_{i-1}(f_2)} n^{f_2} (\tilde{\alpha}, \alpha)   \sum_{\omega \in Cr_{i-2} (f_1) }n^{f_1} (\alpha, \omega) \omega.
 \end{equation}

We can check that the above assumptions 
occur when we cancel the pair of critical 
simplices $(\tilde{\sigma}, \tilde{\alpha})$.
\begin{definition}\label{birth-death transition}
Under the conditions specified, 
we call \textit{$h$ a birth transition}
and \textit{$g$ a death transition}.
Particularly, 
when $k=1$, $h$ and $g$ are called  \textit{the birth transition of cusp type}
and \textit{the death transition of cusp type},
respectively.
\end{definition}

Note that in smooth case, $g$ corresponds to the 
\textit{cusp degeneration} 
in Cerf theory,
where neighboring two critical points are cancelled at a cusp \cite{cerf, Arnold}.
See also Section 4.3.

\begin{theorem}
\label{A-degeneration}
Let $f_1, f_2: K\to \mathbb{R}$ be discrete Morse functions, and
$h$ and $g$ be the corresponding connectedness homomorphisms. 
If $h$ and $g$ are a  birth and death transition, respectively, then both
$h$ and $g$ are chain maps.
\end{theorem}

\begin{proof}

To prove $h$ is a chain map,  we need to show
for any  $\delta \in Cr_q(f_1)$,
\begin{equation}\label{h is a chain map}
\partial_q^{f_2} \circ h_q (\delta)= h_{q-1} \circ \partial_q^{f_1} (\delta). \tag{$*$}
\end{equation}
We note that
when $q\neq i-1, i$,
$\delta$ is also $f_2$-critical,
and $h_q(\delta)= g_q(\delta)= \delta.$
Thus, if $q\neq i-1,i,i+1$,
then by Equation \ref{1},
$\partial_q^{f_2} \circ h_q=\partial_q^{f_1}=  \id \circ \partial_q^{f_1} = h_{q-1} \circ \partial_q^{f_1}.$

When $q=i+1$, for any $\tau \in Cr_{i+1}(f_1)$,
by  Equations \ref{1} and \ref{2},
we see
\begin{align*}
\partial_{i+1}^{f_2} \circ h_q (\tau)&= \partial_{i+1}^{f_2}(\tau)= \sum_{\sigma\in Cr_{i}(f_1)} n^{f_1} (\tau, \sigma) \sigma + \sum_{\sigma\in Cr_{i}(f_1)} n^{f_1} (\tau, \sigma) \cdot n^{f_1} (\sigma, \tilde{\sigma}) \tilde{\sigma}.\\
&= h_i(\sum_{\sigma\in Cr_{i}(f_1)} n^{f_1} (\tau, \sigma) \sigma)=h_{i} \circ \partial_{i+1}^{f_1} (\tau).
\end{align*}

Also,
when $q=i$,  for
any $\sigma \in Cr_{i} (f_1)$,
Equations \ref{3} and \ref{4} imply 
\begin{align*}
\partial_i^{f_2} \circ h_i(\sigma)&=  \partial_i^{f_2}(\sigma + n^{f_1}(\sigma, \tilde{\sigma})  
\tilde{\sigma} )\\
&=\partial_i^{f_1}(\sigma) - n^{f_1}(\sigma, \tilde{\sigma})   (k \tilde{\alpha} + \sum_{\alpha \in Cr_{i-1}(f_1)   }    n^{f_2} (\tilde{\alpha}, \alpha   )  \alpha     ) \\
&+
n^{f_1}(\sigma, \tilde{\sigma})( k \tilde{\alpha}  +   \sum_{\alpha \in Cr_{i-1}(f_1)}  n^{f_2} (\tilde{\alpha},  \alpha) \alpha )\\
&=\partial_i^{f_1}(\sigma)=h_{i-1} \circ \partial^{f_1}_i(\sigma).
\end{align*}

Finally,
when $q=i-1$,
\ref{h is a chain map}
holds trivially
because 
$h_{i-1}$ and  $h_{i-2}$ are identities,
and 
$\partial^{f_1}_{i-1}=\partial^{f_2}_{i-1}$.
Therefore, 
$h$ is a chain map.

We omit the  parallel proof  applied to the connectedness homomorphism $g$
to prove that $g$ is a chain map.

\end{proof}

With regard to the birth-death transitions,
we remark the following statements.
\begin{itemize}

\item $g\circ h = \id.$ Also, $h\circ g(\delta_2) = \id$, 
for any $\delta_2 \neq \tilde{\sigma} $ or $\tilde{\alpha}$.

\item $h$ and $g$ are chain-homotopy inverses to each other, 
that is, 
$g\circ h$ and $h\circ g$ are chain homotopic to the identities.

\item In Definition 4.1, for convenience,
we assumed that any $f_2$-critical simplex $\delta_2$ distinct from 
$\tilde{\sigma} $ or $\tilde{\alpha}$ 
is also $f_1$-critical, implying  $h_q=g_q=\id_q$,
for $q\neq i$.
However, these identical simplices  can be  
extended to a pair of $f_1$ and $f_2$-critical simplices, 
which are faithful to each other.
See Example 2 in Section 4.4.
\end{itemize}

Next, we use the  connectedness homomorphisms
to define  the \textit{function connectedness} between 
two discrete Morse functions.

\begin{definition}
\label{connectedness}
We say  $f_1$ \textit{is  connected to} $f_2$
if both $h$ and $g$ are isomorphisms,
denoted $f_1 \leftrightarrow f_2$.

\end{definition}


Note that neither the simplex connectedness nor the
function connectedness  is an
equivalence relation
due to the absence of the transitive property.
To illustrate, 
let $f_1,f_2,f_3: K \to \mathbb{R}$ be discrete Morse functions and 
$\sigma_i \in Cr_q (f_i),$ $i= 1,2,3,$ 
represent critical simplices corresponding to each function. 
 The presence of  connectedness 
 $\sigma_1 \to \sigma_2$  and $\sigma_2 \to \sigma_3$ 
 does not necessarily imply 
 the strong connectedness $\sigma_1 \to \sigma_3$.
Alternatively, 
expressed in terms of the connectedness coefficient,
$n^{f_1}(\sigma_1, \sigma_3) \neq n^{f_1}(\sigma_1, \sigma_2) \cdot n^{f_2}(\sigma_2, \sigma_3).$

\begin{definition}
Let $f_1, f_2, \ldots,  f_n : K \to \mathbb{R}$
be a sequence of 
discrete Morse functions,
and $C_{\ast}^{f_1},C_{\ast}^{f_2} ,\ldots, C_{\ast}^{f_n}$
be  the corresponding 
discrete Morse complexes.
If for each $i= 1, 2, \cdots, n-1$, 
the connectedness homomorphisms 
$h^i_{\ast}: C_{\ast}^{f_i} \to C_{\ast}^{f_{i+1}}$
is either an isomorphism or a  
birth/death
 transition,
then we call the 
sequence 
of discrete Morse functions 
a \textit{birth-death transition sequence},
and any two 
discrete Morse functions $f_j$ and $f_k$
are called 
\textit{transitively connected},
denoted $f_j \Tsim  f_k$.
\end{definition}

\begin{proposition}
Transitive connectedness $\Tsim$ is an equivalence relation.
\end{proposition}


\subsection{Comparison of Degenerations in Smooth and Discrete Cases. }
In \cite{laudenbach}, 
Laudenbach  studies the modification of a smooth Morse complex
near a degenerated critical point.
Notably, 
under certain conditions, 
Laudenbach establishes a composition
 of homomorphisms 
 connecting two 
 smooth Morse complexes, 
 obtained by the Morse functions
 before and after the degeneration.

Figure \ref{discretization}
presents a diagram  
illustrating the  relationship between the smooth degeneration and the   connectedness homomorphisms $h$ and $g$.
Within this context, 
let $M$ be a smooth manifold, $\phi_1, \phi_2: M \to \mathbb{R}$ 
be Morse functions. 
The homomorphisms 
\circled{1} are constructed  in Laudenbach \cite[Section f]{laudenbach}
to represent the transformation before and after the degeneration;
and 
the discretization homomorphisms 
 \circled{2} are first derived by Gallais \cite[Theorem 3.1]{Gallais} and then improved by
 Benedetti \cite[Main Theorem A]{Benedetti}, as noted in
Proposition \ref{subdivision} below.
\begin{figure}[h]
\begin{center}

\begin{tikzcd}
S_q^{\phi_1}(M) \arrow[rrr, "\circled{1}"] \arrow[dd, "\circled{2}"]                &  &  & S_q^{\phi_2}(M) \arrow[dd,"\circled{2}"]  \arrow[lll, shift left=2]                                        \\
                                                   &  &  &                                                                  \\

C_q^{f_1}(T) \arrow[rrr, "h_q"] &  &  & C_q^{f_2}(T) \arrow[lll, "g_q", shift left=2]
\end{tikzcd}
\end{center}

\caption{Comparison of the smooth and discrete cases.}
\label{discretization}
\end{figure}

\begin{proposition}[Main Theorem A of Benedetti  \cite{Benedetti}]
\label{subdivision}
Let $M$ be a smooth manifold
and $\phi:M \to \mathbb{R}$ be a generic Morse function.
For any PL triangulation $T$ of $M$,
there exist an integer $r$ and a discrete Morse function $f$ on 
the $r$-th derived subdivision $T^r$ of $T$
such that  
the smooth Morse complex is isomorphic to the discrete Morse complex obtained
by
$T^r$ and $f$.
\end{proposition}

\begin{proposition}\label{commutative}
The diagram in Figure \ref{discretization}
is  commutative.
\end{proposition}
Proposition \ref{commutative}
follows from the followings: 
homomorphisms 
\circled{1} are chain maps;
homomorphisms 
\circled{2} are isomorphisms;
and by
Theorem \ref{A-degeneration},
$h$ and $g$ are chain maps.
For the detailed constructions of \circled{1}, \circled{2} and the proofs,
one may refer to the papers.

Hence,
a relationship 
between smooth degenerations 
and discrete birth-death transitions 
can be established 
through the composition of  the homomorphisms.
With regard to the relationship,
we make the following remarks:
\begin{itemize}

\item  The reverse process of this construction is not universally feasible.
That  is,  for a given  triangulation $T$ of some smooth manifold $M$
and a discrete Morse function $f: T\to \mathbb{R}$,
it may not be possible to associate
$f$ to a smooth Morse function on $M$ such that
the  Morse complex and the discrete Morse complex are isomorphic.
(See \cite[Main Theorem B]{Benedetti})

\item Our Definition \ref{birth-death transition} employs 
general simplicial complexes $K$,
thus containing the cases of the results obtained by
the diagram of Figure 3,  in which $T$ is a triangulation
of some smooth manifold. 
This is exemplified   in Example 2 of Section 4.3.
\end{itemize}

  Additionally, we remark that some useful discussions about 
the simplex connectedness 
of subdivision 
are addressed \cite[Section 4]{birth and death}.

\subsection{ Examples of  Birth-Death Transition.}

For a smooth function $\phi$
on a manifold with a 
parameter $\lambda$,
the cusp degeneration 
can be described by the normal form:
$$\phi(\mathbf{x}, \lambda)= x_1^3 + \lambda x_1 \pm x_2^2 \pm \cdots \pm x_m^2.$$
Thus,  to 
study this bifurcation generally,
it is essential to understand
the one-dimensional case.
We consider the 
corresponding discrete settings as follows.

\subsection*{Example 1.}
The following example illustrates the modification of a discrete Morse complex near a birth-death point through the connectedness homomorphism. Figure \ref{continuous functions} depicts two continuous functions, 
$\phi_1$ and $\phi_2$, defined over the same interval 
$I=[-5,5].$

$\phi_1$ and $\phi_2$ are used to denote the transformation 
before and after a cusp degeneration.
Besides the endpoints of the interval,
function $\phi_1$ consists of two critical points;
and
$\phi_2$ consists of four critical points.

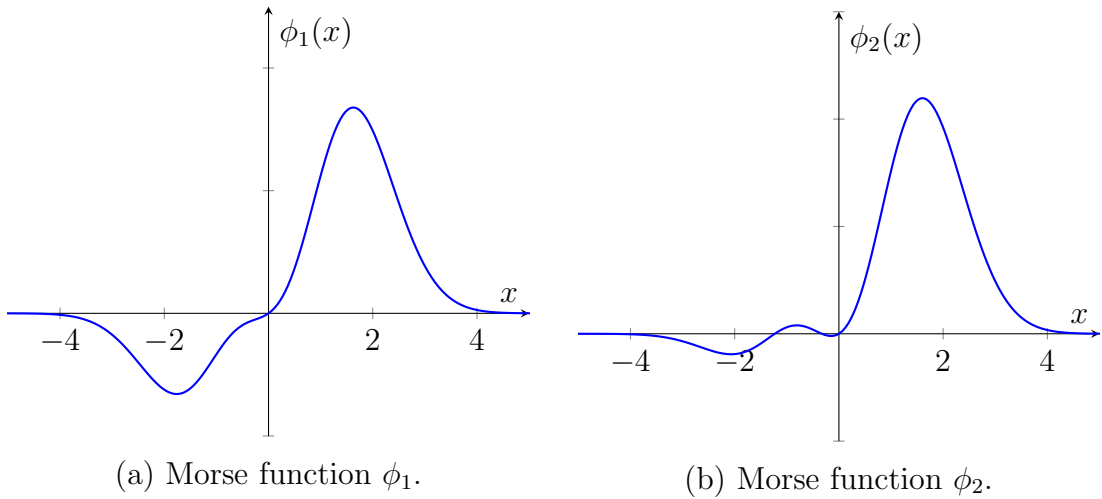
\begin{figure}[!h]
\label{smooth}
\centering
\begin{minipage}{.5\textwidth}
  \centering
 \begin{tikzpicture}

\begin{axis}[
    axis lines=middle,
    xlabel={$x$},
    ylabel={$\phi_1(x)$},
    xmin=-5, xmax=5,
    ymin=-10, ymax=25, 
    samples=400,
    yticklabels={},  
]
\addplot[blue, thick, domain=-10:10] {exp(-x^2/2)*(x^4/2 +9*x^3 + 6*x^2 +3*x)};
\end{axis}
\end{tikzpicture}

(a) Morse function $\phi_1$.
\end{minipage}%
\begin{minipage}{.5\textwidth}
  \centering
 \begin{tikzpicture}

\begin{axis}[
    axis lines=middle,
    xlabel={$x$},
    ylabel={$\phi_2(x)$},
    xmin=-5, xmax=5,
    ymin=-5, ymax=15, 
    samples=400, 
    yticklabels={},  
]
\addplot[blue, thick, domain=-10:10] {exp(-x^2/2)*(x^4/2 +4.5*x^3 + 6*x^2 +1.5*x)};
\end{axis}
\end{tikzpicture}

(b) Morse function $ \phi_2$.
\end{minipage}
\caption{Two Morse functions on the  interval $I$. }
\label{continuous functions}
\end{figure}

 Figure \ref{discrete functions} 
 displays the triangulation of the interval $I$
  into a one-dimensional simplicial complex 
  $K$.
Also,  continuous functions $\phi_1$ and $\phi_2$
are represented by 
discrete  functions $f_1$ and $f_2$
in (a) and (b) of Figure \ref{discrete functions}.
One can  check that both of them are discrete Morse functions.

For discrete Morse function $f_1$,
there are two $0$-dimensional critical 
simplices and 
one $1$-dimensional critical 
simplex,
labeled as $v_1^1, v_1^2$ and $ e_1^1$.
For discrete Morse function $f_2$,
there are three $0$-dimensional critical 
simplices and 
two $1$-dimensional critical 
simplices,
labeled as $v_2^1, v_2^2, v_2^3, 
 e_2^1$ and $e_2^2$.

Next, we  consider 
the homomorphisms 
$h_q: C_q^{f_1}(K) \to C_q^{f_2}(K)$
and 
$g_q: C_q^{f_2}(K) \to C_q^{f_1}(K),$
for  $q=0,1$.
One can check by the definition,
$h$ and $g$ are birth and death transitions corresponding to the $f_2$-critical
simplices pair $(e_2^2,v_2^2)$.
Simple computations in mod 2 allow us to 
conclude that 
$h_1(e_1^1)=e_2^1+ e_2^2$;
$h_0(v_1^1)=v_2^1$;
$h_0(v_1^2)=v_2^3$.
Boundary computations confirm that 
$h$ is a chain map:
$$\partial^{f_2} \circ h_1 (e_1^1)= v_2^1+v_2^3= h_0\circ \partial^{f_1} (e_1^1).$$

Additionally,  similar computations for 
$g$ confirm that $g$ is also a chain map.

\begin{figure}[!h ]

\begin{minipage}{\textwidth}

  \centering
\begin{tikzpicture}[>=stealth, thick, scale=1.25,
  decoration={markings, mark=at position 0.5 with {\arrow{>}}}
]
    \tikzstyle{vertex}=[circle, fill=black, inner sep=0pt, minimum size=5pt]
    \tikzstyle{redvertex}=[circle, fill=red, inner sep=0pt, minimum size=5pt]

    \foreach \x in {1,2,3,4,5,6,7,8,9,10} {
        \node[vertex, label=below:{}] (v\x) at (\x,0) {};
    }

    \node[redvertex, label=below:{$v_1^2$}] (v10) at (10,0) {};
    \node[redvertex, label=below:{$v_1^1$}] (v3) at (3,0) {};
	
    \draw[->] (v1) -- (v2);
        \draw[->] (v2) -- (v3);
    \draw[->] (v4) -- (v3);
    \draw[->] (v5) -- (v4);
    \draw[->] (v6) -- (v5);
    \draw[->] (v7) -- (v6);
    \draw[->] (v8) -- (v9);
    \draw[->] (v9) -- (v10);

    \draw[red] (v7) -- (v8) node[midway, above] {$e^1_1$}; 

    \draw (v1) -- (v2);
    \draw (v3) -- (v4);
    \draw (v4) -- (v5);
    \draw (v5) -- (v6);
    \draw (v6) -- (v7);
    \draw (v8) -- (v9);
    \draw (v9) -- (v10);

\end{tikzpicture}

(a) Discrete Morse function $f_1$.\\

\end{minipage}%

\begin{minipage}{.82\textwidth}
  \centering
\begin{tikzpicture}[>=stealth, thick, scale=1.25,
  decoration={markings, mark=at position 0.5 with {\arrow{>}}}
]
    \tikzstyle{vertex}=[circle, fill=black, inner sep=0pt, minimum size=5pt]
    \tikzstyle{redvertex}=[circle, fill=red, inner sep=0pt, minimum size=5pt]

    \foreach \x in {1,2,3,4,5,6,7,8,9,10} {
        \node[vertex, label=below:{}] (v\x) at (\x,0) {};
    }

   \node[redvertex, label=below:{$v_2^1$}] (v3) at (3,0) {};
    \node[redvertex, label=below:{$v_2^3$}] (v10) at (10,0) {};
   \node[redvertex, label=below:{$v_2^2$}] (v6) at (6,0) {};

    \draw[->] (v1) -- (v2);
        \draw[->] (v2) -- (v3);
    \draw[->] (v4) -- (v3);
    \draw[->] (v5) -- (v4);
    \draw[->] (v7) -- (v6);
    \draw[->] (v8) -- (v9);
    \draw[->] (v9) -- (v10);

    \draw[red] (v7) -- (v8) node[midway, above] {$e^1_2$}; 
 \draw[red] (v5) -- (v6) node[midway, above] {$e^2_2$}; 

\end{tikzpicture}
(b) Discrete Morse function $f_2$.
\end{minipage}
\caption{Two discrete Morse functions on the triangulation of $I$.   }
\label{discrete functions}
\end{figure}
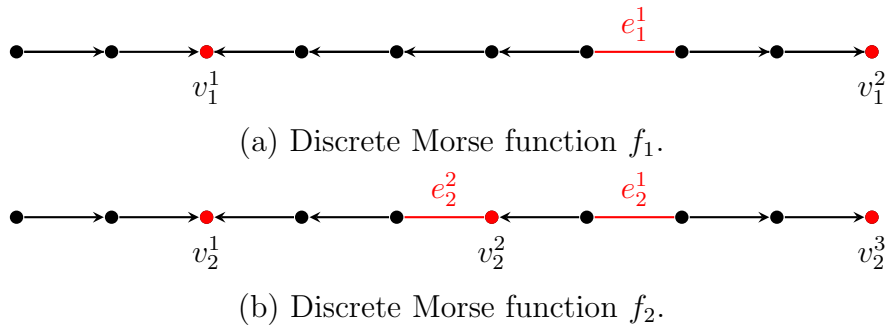

\subsection*{Example 2.}
In Proposition 4.5, 
we showed that 
the birth-death transitions on a triangulation $T$ of a
smooth manifold $M$ 
can be realized
through the composition of homomorphisms mentioned in 
Figure 3.
However, 
Theorem 4.2 contains this 
result as $K$ in Theorem 4.2 represents a general
simplicial complex.
We illustrate this fact with the following 
example.

Figure 6
visually depicts a birth-death transition of the cusp type
on a 
simplicial complex $K$,
which is not a triangulation of any smooth manifold.
Note that 
although $e_1^1\neq  e_2^1$, $e_1^1$ and $e_2^1$ are faithful to each other;
and 
$v_1^1$ and $v_2^2$ are faithful to each other, as well. 
We let $(e_2^2, v_2^1)$
be the pair of $f_2$-critical simplices to be cancelled.  
Then, 
$h$ and $g$ are  birth-death transitions.

That $h$ is a chain map 
follows from $f_1$ is optimal.
For $g$, 
we compute that
$g_0(v_2^1)=g_0(v_2^2)=v_1^1$;
$g_1(e_2^1)= e_1^1$,
$g_1(e_2^2)= 0$.
Furthermore,
$\partial_1^{f_2}(e_2^1)= 0$
and
$\partial_1^{f_2}(e_2^2)= v_2^1 + v_2^2$ 
in mod $2$.
Consequently, 
$g_0\circ \partial_1^{f_2}(e_2^1)= 0 = \partial^{f_1}_0 \circ g_1(e_2^1);$
and
$g_0\circ \partial_1^{f_2}(e_2^2)= 0 = \partial^{f_1}_0 \circ g_1(e_2^2).$
Hence, 
$g$ is also a chain map.

Note that the choice of 
critical simplices for cancellation is not necessarily
unique and depends on the orientation of $K$.
In this scenario, we can choose to cancel
either  the pair 
$(e_2^2,  v_2^1)$
or 
$(e_2^2,  v_2^2)$.
However, 
regardless of the choice made,
the remaining $0$-dimensional 
simplex remains strongly connected to 
$v_1^1$, and
$h$ and $g$ remain to 
be birth and death transitions.

Compared with Example 1,
the simplicial complex in Example 2 is
not locally Euclidean.
However, 
our focus should 
rather be on the gradient vector fields defined on this simplcial complex.
We can decompose the 
gradient vector fields 
into a set of gradient paths (See \cite[Example 4.9]{no.1} as an example).
Thus,
as a basic fact of discrete Morse theory \cite[Theorem 3.5]{Forman_guide},
any  gradient path does not 
include cycles,
thus ensuring each  gradient path is ``locally Euclidean.
Consequently,
the structure of gradient vector fields in Example 2
can be described as a union of 
the type of triangulations in Example 1.

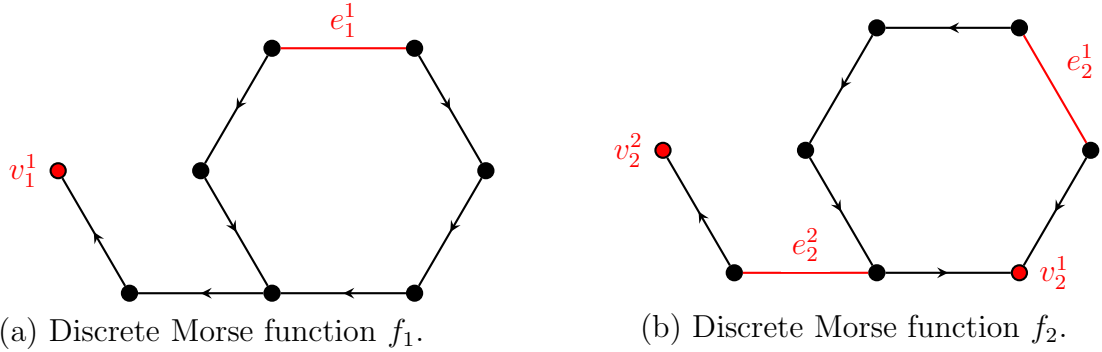
\begin{figure}

\begin{minipage}{.5\textwidth}

\begin{tikzpicture}[>=stealth, thick, scale=1.25, decoration={markings, mark=at position 0.5 with {\arrow{>}}}]
  \foreach \x in {1,2,...,6}{
    \node[draw, circle, fill=black, inner sep=2pt, label=above:] (vertex\x) at ({360/6 * (\x - 1)}:1.5cm) {};

  }

  \path let \p1 = ($(vertex1) - (vertex2)$), \n1 = {veclen(\x1,\y1)} in
  node[draw, circle, fill=black, inner sep=2pt, label=left:] (vertex7) at ($(vertex5) - (\n1, 0)$) {} 
  node[draw, circle, fill=red, inner sep=2pt, label={[text=red]left:$v_1^1$}] (vertex8) at ($(vertex4) - (\n1, 0)$) {};


\draw[postaction={decorate}] (vertex2) -- (vertex1);
\draw[postaction={decorate}] (vertex1) -- (vertex6);
\draw[postaction={decorate}] (vertex6) -- (vertex5);
\draw[postaction={decorate}] (vertex3) -- (vertex4);
\draw[postaction={decorate}] (vertex4) -- (vertex5);
\draw[red] (vertex2) -- (vertex3) node[midway, above, fill=white] {$e_1^1$} ;
\draw[postaction={decorate}] (vertex7) -- (vertex8);
\draw[postaction={decorate}] (vertex5) -- (vertex7);

\end{tikzpicture}

(a) Discrete Morse function $f_1$.
\end{minipage}%
\begin{minipage}{.5\textwidth}

\centering

\begin{tikzpicture}[>=stealth, thick, scale=1.25, decoration={markings, mark=at position 0.5 with {\arrow{>}}}]
  \foreach \x in {1,2,...,6}{
   \ifnum\x=6
      \node[draw, circle, fill=red, inner sep=2pt, label={[text=red]right:$v_2^1$}] (vertex\x) at ({360/6 * (\x - 1)}:1.5cm) {};
      \else
    \node[draw, circle, fill=black, inner sep=2pt, label=above:] (vertex\x) at ({360/6 * (\x - 1)}:1.5cm) {};
    \fi
  }

  \path let \p1 = ($(vertex1) - (vertex2)$), \n1 = {veclen(\x1,\y1)} in
  node[draw, circle, fill=red, inner sep=2pt, label={[text=red]left:$v_2^2$}] (vertex8) at ($(vertex4) - (\n1, 0)$) {} 
  node[draw, circle, fill=black, inner sep=2pt, label=left:] (vertex7) at ($(vertex5) - (\n1, 0)$) {};

\draw[postaction={decorate}] (vertex7) -- (vertex8);
\draw[red] (vertex5) -- (vertex7) node[midway, above, fill=white] {$e_2^2$} ;
\draw[postaction={decorate}] (vertex2) -- (vertex3);

\draw[postaction={decorate}] (vertex1) -- (vertex6);
\draw[postaction={decorate}] (vertex5) -- (vertex6);
\draw[postaction={decorate}] (vertex3) -- (vertex4);
\draw[postaction={decorate}] (vertex4) -- (vertex5);
\draw[red] (vertex1) -- (vertex2) node[midway, above right, fill=white] {$e_2^1$} ;

\end{tikzpicture}

(b) Discrete Morse function $f_2$.
\end{minipage}

\label{transition example}
\caption{A birth-death transition on a simplicial complex that is not 
a triangulation of a smooth manifold.}
\end{figure}

\subsection{Birth-Death Transitions on the Complex of Discrete Morse Functions}

In discrete Morse theory, 
the set of  discrete Morse functions on a simplicial complex
is typically studied by  
defining  an equivalence relation:
discrete Morse functions $f_1$ and $f_2$ are considered \textit{equivalent}
if their corresponding 
gradient vector fields $V_1$ and $V_2$ 
are identical as sets.
This definition ignores 
the numerical values  of the 
discrete Morse functions,  instead focusing 
 on  the order of values within
each simplex's neighborhood.

Also, 
the collection of gradient vector fields form a simplicial complex,
known as the
\textit{ complex of discrete Morse functions},
as initially studied in \cite{chari}.
Furthermore, when $K$ is a graph,
the simplicial complex is the
complex of rooted forests, 
 introduced in \cite{Kozlov}.

More precisely,
given a simplicial complex $K$,
we define the \textit{complex of discrete Morse functions} $\mathcal{M}(K)$
on $K$ as follows.
We say a \textit{primitive vector field} of $K$
is a set  $P$ containing one  pair of simplices 
$\alpha$  and  $\sigma$ of $K$ such that 
$\alpha \prec \sigma$ and their codimension is $1$.
We let $\mathcal{P}(K)= \{P_1, P_2, \ldots, P_n  \}$ 
be the set of all primitive vector fields of $K$.

We construct $\mathcal{M}(K)$
as follows:
\begin{itemize}
\item The set of $0$-simplices $\mathcal{M}^{0}(K)= \mathcal{P}(K)$;
\item $( P_{i_0}, P_{i_1}, \ldots, P_{i_k} )$ forms a $k$-simplex in 
$ \mathcal{M}^k(K)$
if $V= \{ P_{i_0}, P_{i_1}, \ldots, P_{i_k} \}$ 
is a gradient vector field on $K$.

\end{itemize}

Note that the collection of primitive vector fields $V$
is a gradient vector field
if and only if
\begin{enumerate}[label=(\subscript{C}{{\arabic*}})]
\item each simplex in $V$ appears at most  once; and
\item $V$ is acyclic.
\end{enumerate}
The simplices that do not appear in $V$ are called \textit{critical}.

It is not hard to see that
	each simplex in $\M(K)$ 
corresponds to a class of discrete Morse functions, identified
by  those  having the same gradient vector field.
Conversely, 
each gradient vector field corresponds to 
a simplex of $\M(K)$, as well.
However, technically,  we have to eliminate 
the case of \textit{empty simplex} $\emptyset$
of $\M(K)$,
because the empty simplex corresponds to 
the
 gradient vector field $V$ where every simplex in $V$ is critical. 
Figure \ref{complex example} 
illustrates  the complex of  discrete Morse function $\M(K)$ of $K$.

In this subsection,
we explain the relationship between birth-death transitions 
and the
complex of discrete Morse functions.

\begin{lemma}
Let $f_1, f_2: K\to \mathbb{R}$
be discrete Morse functions
and $V_1, V_2$ be the corresponding gradient vector
field, respectively.
Suppose that 
$h_q: C_q^{f_1} \to C_q^{f_2}$ 
and 
$g_q: C_q^{f_2} \to C_q^{f_1}$ 
be the connectedness homomorphisms.
Then, $V_1=V_2$
if and only if $h=g=\id.$
\label{f_1=f_2}
\end{lemma}

\begin{proof}

Suppose $V_1=V_2$.
It is easy to see that for any $\sigma_1$,
$h_q(\sigma_1)=\sigma_1=g_q(\sigma_1)$,
because there is no gradient path 
between two $q$-dimensional critical simplices.

The proof of the converse is straightforward.
\end{proof}

Next, 
we describe the birth-death transition 
on the complex of discrete Morse functions.
By $\overline{\sigma}_1, \overline{\sigma}_2 \in \M(K)$,
we denote the simplex in $\M(K)$
corresponding to $V_1,  V_2$,
respectively.

\begin{lemma}
The following statements hold true.
\begin{enumerate}
\item $h$ is a birth transition if and only if
 $\overline{\sigma}_1 \succ \overline{\sigma}_2$,
and $\dim \overline{\sigma}_1- \dim \overline{\sigma}_2=1$.

\item $h$ is a death transition if and only if 
 $\overline{\sigma}_1 \prec \overline{\sigma}_2$,
and $\dim \overline{\sigma}_1- \dim \overline{\sigma}_2=-1$.
\end{enumerate}

\end{lemma}

\begin{proof}
We prove statement (1).
Suppose that 
$h$ is a birth transition,
and a pair of simplices $\alpha \prec \sigma$
are born by $h$.
Then, 
$V_2 = V_1 \setminus \{ \alpha, \sigma \}$.
By the definition of complex of discrete Morse functions,
$\overline{\sigma}_1$ is a subcomplex of 
$\overline{\sigma}_2$,
and $\dim \overline{\sigma}_1- \dim \overline{\sigma}_2=-1$.

The proof of the converse follows from 
Definition 4.1.

(2) follows from a similar proof.
\end{proof}

\begin{proposition}
If $\overline{\sigma}_1$ and $\overline{\sigma}_2$
are in the same component of $\M(K)$,
then $f_1 \Tsim f_2$.

\end{proposition}

\begin{proof}
Suppose that 
$\overline{\sigma}_1$ and $\overline{\sigma}_2$
are in the same component of $\M(K)$.
Then, 
there is a path consisting of 
simplices with adjacent dimensions 
connecting 
$\overline{\sigma}_1$ and $\overline{\sigma}_2$.
By Lemma 4.7,
we can construct a sequence of 
birth/death transitions 
from the simplicial path.
\end{proof}

Note that the converse of Proposition 4.8 is ``almost"
true if we exclude the $\emptyset$ case in the sequence.

\begin{example}
Figure \ref{complex example} $(b)$
illustrates 
an example of the complex of discrete 
Morse functions 
obtained from the simplicial complex $K$ in $(a)$.

We let  $\overline{\sigma}_1=(a,e_1)$ and 
$\overline{\sigma}_2=((a,e_1), (b,e_2))$.
Thus, the corresponding gradient vector fields
$V_1=\{ (a,e_1) \}$
and 
$V_2=\{(a,e_1), (b,e_2) \}$.
Hence, 
one can check that
the 
connectedness homomorphisms
$h$ and $g$ 
are a death and birth transition,
respectively.

\begin{figure}[h]

\begin{minipage}{.5\textwidth}

\begin{tikzpicture}[scale=1.5, thick]
  \node[draw, circle, fill=black, inner sep=2pt, label=left:a] (vertex1) at (0, 0) {};
  \node[draw, circle, fill=black, inner sep=2pt, label=right:b] (vertex2) at (2, 0) {};
  \node[draw, circle, fill=black, inner sep=2pt, label=above:c] (vertex3) at (1, 1.73) {}; 

  \draw (vertex1) -- (vertex2) node[midway, below, yshift=-5pt, fill=white] {$e_1$};
  \draw (vertex2) -- (vertex3) node[midway, right, xshift=5pt, fill=white] {$e_2$};
  \draw (vertex3) -- (vertex1) node[midway, left, xshift= -5pt, fill=white] {$e_3$};
\end{tikzpicture}

(a) Simplicial complex  $K$.
\end{minipage}%
\begin{minipage}{.5\textwidth}

\centering
\begin{tikzpicture}[scale=2, thick]
  \node[draw, circle, fill=black, inner sep=2pt, label=above left:{$(a,e_1)$}] (ae1) at (0.1,1) {};
  \node[draw, circle, fill=black, inner sep=2pt, label=above right:{$(c,e_2)$}] (ce2) at (2,1) {};
  \node[draw, circle, fill=black, inner sep=2pt, label=below right:{$(a,e_3)$}] (ae3) at (1.9,0) {};
  \node[draw, circle, fill=black, inner sep=2pt, label=below left:{$(b,e_2)$}] (be2) at (0,0) {};
  \node[draw, circle, fill=black, inner sep=2pt, label=below:{$(c,e_3)$}] (ce3) at (0.7,0.6) {};
  \node[draw, circle, fill=black, inner sep=2pt, label=above:{$(b,e_1)$}] (be1) at (2.7,0.6) {};
  
  \draw (ae1) -- (be2);
  \draw (ae1) -- (ce3);
  \draw (be2) -- (ce3);
  
   \draw (ae3) -- (be1);
  \draw (ae3) -- (ce2);
  \draw (be1) -- (ce2);
  
     \draw (ae3) -- (be2);
  \draw (ae1) -- (ce2);
  \draw (be1) -- (ce3);
\end{tikzpicture}

(b) Discrete Morse function complex $\mathcal{M}(K)$
\end{minipage}

\caption{An example of discrete Morse function complex.}
\label{complex example}
\end{figure}
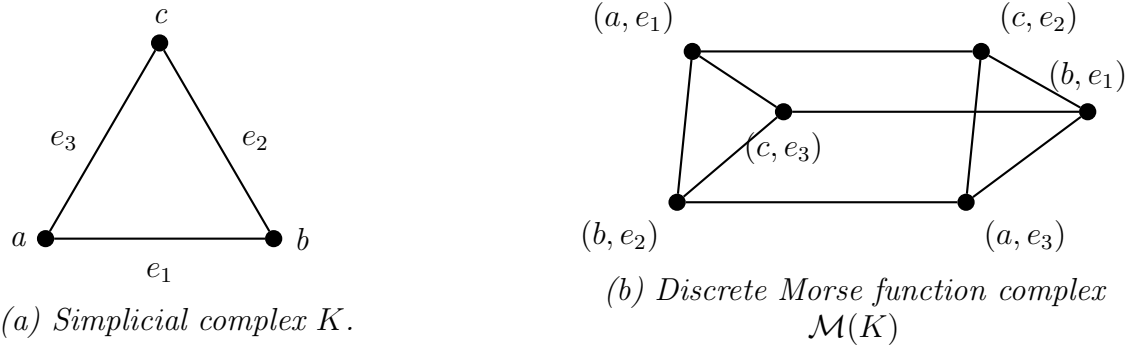

\end{example}

From the discussions above,
we can see that a birth/death transition $h$ corresponds to
a unique pair of simplices $(\overline{\sigma}_1, \overline{\sigma}_2)$
of $\M(K)$.
Hence, we can further construct the gradient vector fields 
on $\M(K)$
by regarding each $h$ as a primitive  vector field on $\M(K)$.
	When a collection of  primitive vector fields obtained by 
a family of birth/death transitions 
$\{h_i\}$ satisfies the conditions ($C_1$) and  ($C_2$) 
mentioned above,
the collection form a \textit{gradient vector field} on $\M(K)$.
This construction is 
a foundational step for the 
discrete Morse theory of the complex of discrete Morse functions.

Although several works \cite{chari, homotopy_type, homotopy_type 2}
have investigated the homotopy type 
 of $\M(K)$,
computing the homotopy type  of $\M(K)$
remains a challenging task.
One primary reason is that as
$K$ increases in size,
$\M(K)$
becomes significantly more complicated.
However, the gradient vector field defined on $\M(K)$
could be be utilized here in 
terms of discrete Morse functions on $K$ and birth-death transitions. 
Since a connected 
simplicial complex $K$ is uniquely determined by 
$\M(K)$ \cite{determined},
this may lead to a  more convenient 
approach to investigating the homotopy type of $\M(K)$.

One of our 
long-term goal of this
series of work is 
to establish a
discrete version of 
Cerf theory,
which is introduced in \cite{cerf}.
The core of Cerf theory
 is the Cerf diagram, which visually represents the critical
  value changes 
  of a generic one-parameter family of functions on a manifold.
  This is pivotal for understanding bifurcations
   and the structural stability of function spaces. 
   The theory notably applies to gradient vector fields 
   and their trajectories, studying how critical points
    can be created or 
    eliminated through perturbations, known as births and deaths in the context of 
smooth   Morse theory.

Some discrete counterparts have been considered, 
e.g., Chari-Joswig \cite{chari}, King-Knudson-Mramor Kosta \cite{birth and death}.
Also recently,
there appeared other attempts for discussing 
a discrete analogy to the smooth Cerf theory, such as Br\"{u}ggemann \cite{hyperplane}.
Our approach 
offers a distinct advantage 
through the concepts of simplex connectedness and function connectedness, 
which may lead to more 
direct geometric understanding:
by  subdivisions of the gradient vector fields,
we describe 
a discrete analogy of degenerations of 
degenerated critical simplices \cite{no.3}.

\end{document}